\documentclass[11pt]{amsart}
\usepackage{amsmath,amssymb,amscd}
\usepackage{amsmath,amssymb}

\usepackage{amsxtra, amsmath}
\usepackage{amssymb, amscd}
\usepackage{graphicx}

\newcommand\NN{\mathbb N}

\newcommand\tr{\operatorname{tr}}

\theoremstyle{plain}
\newtheorem{thm}{Theorem}[section]

\newtheorem{prop}[thm]{Proposition}
\newtheorem{cor}[thm]{Corollary}

\theoremstyle{definition}

\theoremstyle{remark}
\newtheorem{remark}[thm]{Remark}

\title[\sc Bispectrality and Time-Band-Limiting: Matrix valued polynomials.]{Bispectrality and Time-Band-Limiting: Matrix valued polynomials.}

\author{Gr\"unbaum F. A., Pacharoni  I. and  Zurri\'an I.}
\address{Department of Mathematics, University of California, Berkeley
CA 94705}
\email{grunbaum@math.berkeley.edu}
\address{CIEM-FaMAF, Universidad Nacional de C\'ordoba, C\'ordoba~5000, Argentina}
\email{pacharon@famaf.unc.edu.ar}
\address{CIEM-FaMAF, Universidad Nacional de C\'ordoba, C\'ordoba~5000, Argentina}
\email{zurrian@famaf.unc.edu.ar}

\date{\today}

\thanks{This research was supported in part by CONICET grant PIP
112-200801-01533, SeCyT-UNC and FONDECYT 3160646. The work of the first author is partially supported by AFOSR through FA9550-16-1-0175}



\begin{document}

\begin{abstract} 
	The subject of time-band-limiting, originating in signal processing, is dominated by the miracle that a naturally appearing integral operator admits a commuting differential one allowing for a numerically efficient way to compute its eigenfunctions. Bispectrality is an effort to dig into the reasons behind this miracle and goes back to joint work with H. Duistermaat. This search has revealed unexpected connections with several parts of mathematics, including integrable systems.

	Here we consider a matrix valued version of bispectrality and give a general condition under which we can display a constructive and simple way to obtain the commuting differential operator. 
 Furthermore, we build an operator that commutes with both the time-limiting operator and the band-limiting operators.
\end{abstract}
\maketitle

\section{Introduction}

 The problem of double concentration, i.e. localizing a function both in physical and frequency space cuts across several areas of mathematics, physics and engineering.
 This topic arises in harmonic analysis, signal processing and quantum mechanics.  Highly elaborate bodies of work, such as wavelet theory, spawn from efforts to find a good compromise between these two competing goals.

In some instances this issue gives rise to a sharply posed question as was done (at least implicitly) by C. Shannon, \cite{S}: if you know the frequency components over a band $[-W,W]$ for an unknown signal of finite support in $[-T,T]$, what is the best use you can do of this (noisy) data?
 It is natural to look for the coefficients of an expansion of the unknown signal in terms of the singular functions of the problem. However, one faces a serious computational difficulty: these singular functions are the eigenfunctions of an integral operator with most of its eigenvalues crowded together.

 In a remarkable series of papers written at Bell Labs in the 1960's a mathematical miracle was uncovered, and exploited very  successfully. We refer to it as the "time-band limiting phenomenon".
We are alluding to the
surprising fact that certain naturally appearing integral operators admit second order commuting differential ones.

\medskip

  One of us has been looking for the reason that lies behind this miracle for quite a while and this
search has given rise to what we refer to as the "bispectral problem". In our context this
consists in the search for weights whose orthogonal polynomials are joint eigenfunctions of some differential operator.

\medskip

There is a large number of papers dealing with the relations between these two issues. For a sample, see \cite{G,DG,G4,G5,GLP,Perlstadt,GY}. For surveys of this and related work, see \cite{G1,G5,G7}. We feel that the true reasons behind this remarkable algebraic "accident", see \cite{Slep,Walter},  deserves further study.

The phenomenon of a pair of commuting integral and differential operators plays an important role in at least three areas of applied mathematics: the problem of time-and-band limiting studied by Slepian, Landau and Pollak, see \cite{SLP1,SLP2,SLP3,SLP4,SLP5}, nicely summarized in \cite{Slep,Lan}, the problem of limited angle tomography, see \cite{G6}, and finally in Random Matrix Theory, see \cite{M,TW,TWB}.
For other applications of this work, see \cite{JB,JKS,SD,SDW}. For numerical aspects of this phenomenon, see \cite{ORX}. All of the work mentioned above deals with scalar valued
functions.

\medskip

A much more recent look at the relation between these two topics involves matrix valued orthogonal polynomials, a subject started by M. G. Krein, see \cite{K1,K2}. The papers where this relation has been explored recently are \cite{CG5,GPZ1,GPZ2,CGPZ,CG7}.

\medskip

The list of references given above is pretty complete with one exception. Following \cite{Perlstadt} there is a short and elegant paper by Perline, see \cite{Perline}. One of us was certainly aware of this paper back in the late 1980's, but somehow  did not pay enough attention to it. After the completion of \cite{GVZ_Heun} it was A. Zhedanov who noticed this long forgotten paper and brought it the attention of his coworkers. The very recent paper \cite{GVZ1_Heun} shows that the ideas in \cite{Perline} can be extended to other scenarios.

\medskip

	The aim of this paper is to give a general result on the relation between the bispectral property for matrix valued orthogonal polynomials and the existence of a symmetric operator that commutes with the time-and-band limiting operator and can be used to yield their eigenfunctions.
For any value of the relevant parameters we build explicitly a  second order differential operator $T$ and a tridiagonal  difference operator $L$ that commute with {\bf both} the time-limiting operator and the band-limiting operator. This proves, in a constructive way, the existence of commuting operators for the integral and the difference operators.

This general result, as well as those in \cite{GVZ1_Heun}, is inspired by the construction in \cite{Perline}.

\smallskip

Finally, in Section \ref{ex}, after a brief mention of scalar cases, we use our general results to study some particular examples, all of them in the matrix valued case.

In the first example we extend results previously obtained in \cite{GPZ1,GPZ2};
in the second one we verify a result that was conjectured in \cite{CG7};
in the third example we exploit the power of our construction to give a commuting differential operator for a case where the commuting operator problem was not studied before;
the last example is included to indicate that bispectrality may not always guarantee the existence of a commuting differential operator.

\smallskip

In the scalar case treated in \cite{Perline}, the issue of the use of the commuting differential operator to obtain the eigenfunctions of the integral one was not dealt in detail. In this paper we take the same approach and intend to return to this point at a later time.

\section{Preliminaries}

Let $W(x)$ be an $R\times R$    matrix weight function    in the open interval $(a,b)$ and let  \{$Q_n(x)\}_{n\in\NN_0}$ be a sequence of  matrix orthonormal polynomials with respect to the weight $W(x)$.

\smallskip
The Hilbert spaces $\ell^2(M_R, \NN_0)$ and  $L^2((a,b), W(t)dt)$ are given by the real valued $R\times R$ matrix sequences
$\{C_n\}_{n\in \NN_0}$ such that $\sum_{n=0}^\infty \tr \left( C_n\,C_n^*\right) < \infty$  and all measurable matrix valued functions $f(x)$, $x\in (a,b)$, satisfying $\int_a^b \tr\left(f(x)W(x)f^*(x)\right)dx < \infty $, respectively.
A natural   analog of the Fourier transform is the isometry $F:\ell^2(M_R,\NN_0) \longrightarrow L^2(W)$ given by
$$\{C_n\}_{n=0}^\infty \overset{F}{\longmapsto} \sum_{n=0}^\infty C_n Q_n(x).$$
 In the case when the matrix polynomials are dense in $L^2(W)$,  this map is unitary with the inverse $F^{-1}: L^2(W)\longrightarrow \ell^2(M_R,\NN_0) $ given by
$$ f \overset{F^{-1}}{\longmapsto} C_n=\int_a^b f(x)\,W(x)\, Q^*_n(x) dx.$$

 If we consider the problem of determining a function $f$  from the following (typically noisy) data: $f$ has support on the compact set $[0,N]$ and its Fourier transform $Ff$ is known on a compact set $[a,\Omega]$, one concludes that
we need to compute the singular vectors (and singular values) of the operator $E:\ell^2(M_R,\NN_0)\longrightarrow L^2(W) $ given by
$$E f= \chi_\Omega F \tilde \chi_N f,$$
where $\tilde \chi_N$ is the {\em time limiting  operator} on  $\ell^2(M_R,\NN_0)$ and  $\chi_\Omega$ is the {\em band limiting operator} on  $L^2(W)$.
At level $N$, $\tilde \chi_N$ acts on $\ell^2(M_R,\NN_0)$ by simply setting equal to zero all the components with index larger than $N$. At level $\Omega$, $\chi_\Omega$ acts on $L^2(W)$ by multiplication by the characteristic function of the interval $(a, \Omega)$, $a<\Omega\le b$.

\medskip
 We are thus lead to study the eigenvectors of the operators
$$E^*E= \tilde \chi_N F^{-1} \chi_\Omega F \tilde  \chi_N\qquad \text{ and } \qquad E E^*= \chi_\Omega F \tilde \chi_N F^{-1} \chi_\Omega.$$

The operator $E^*E$, acting on $\ell^2(M_R,\NN_0)$, is just a finite dimensional block-matrix with each  $R\times R$  block given by
\begin{equation}\label{MatrixM}
 (E^*E)_{m,n}= \int_a^\Omega Q_m(x) W(x) Q^{*}_n(x)  dx, \qquad 0\leq m,n \leq N.
\end{equation}

The operator $E E^*$ acts on $L^2((a,\Omega), W(t)dt)$ by means of the integral kernel
\begin{equation}\label{kernel}
  k(x,y)=\sum_{n=0}^N Q_n^*(x)Q_n(y).
\end{equation}
 The integral operator $S=EE^*$ with kernel $k$, defined in \eqref{kernel}, acting on $L^2((a,\Omega), W)$ ``from the right hand side'' is given by
\begin{equation}\label{intoper}
   (fS)(x)=\int_{a}^\Omega f(y)W(y)\big(k(x,y)\big)^*dy.
\end{equation}

\medskip

For general $N$ and $\Omega$ there is no hope of finding the eigenfunctions of $EE^*$ and $E^*E$ analytically. However, there is a strategy to solve this typical inverse problem:
  finding an operator with simple spectrum which would have the same eigenfunctions as the operators $E E^*$ or $E^* E$. This is exactly what Slepian, Landau and Pollak did in the scalar case, when dealing with the unit circle and the usual Fourier analysis. They discovered the following properties:
\begin{itemize}
  \item For each $N$, $\Omega$ there exists a symmetric tridiagonal matrix $L$, with simple spectrum, commuting with $ E^*E$.
  \item For each $N$, $\Omega$ there exists a self-adjoint  second order  differential operator $T$, with simple spectrum, commuting with the integral operator $S=EE^*$.
\end{itemize}


\smallskip
In this paper,   which deals with a continuous-discrete version of the bispectral problem, we give an {\bf explicit construction} of such symmetric operators $L$ and $T$ given certain hypothesis (which is automatically satisfied in the scalar case).

\medskip




 Symmetry for an operator ${T}$ acting on functions defined in $[a,\Omega]$  means that
$$\langle P{T}, Q\rangle_\Omega=\langle P, Q{T} \rangle_\Omega, $$
 for every $P,Q$ in an appropriate dense set of functions,  where
\begin{equation}\label{W-trunc}
	\langle P,Q\rangle_\Omega=\int_{a}^\Omega P(x)W(x)Q^*(x)\, dx.
\end{equation}

From \cite{GPZ2}, given a symmetric differential operator $T$ and an integral operator $S$, with kernel $k$, we have
  \begin{equation}\label{diffkernel}
  T S=ST\qquad \text{ if and only if } \qquad
   \left( k(x,y)^*\right)T_x= (k(x,y)T_y)^*.
  \end{equation}

 (Here we use $T_x$ to stress that $T$ acts on the variable $x$.)

\smallskip

Notice that in principle there is no guarantee that we will find any such ${T}$ except for a scalar multiple of the identity. For the problem at hand, namely the efficient computation of the eigenfunctions of $S$, we need to exhibit a differential operator  ${T}$   whose eigenfunctions are also eigenfunctions of the integral operator $S$. In the scalar case this is guaranteed by asking that ${T}$ should have a simple spectrum. In the matrix valued case the useful requirement on ${T}$ is more subtle and will be analyzed in detail in a future publication.

\section{The symmetric bispectral problem}

We start with a matrix weight $W$ defined in the interval $(a,b)$ and a second order symmetric differential operator $D$ with respect to $W$ of the form
$$D= \partial ^2 F_2 + \partial  F_1 +F_0,$$
with $F_j$ a polynomial of order less than or equal to $j$, for $j=0,1,2.$

 Let $\{R_n\}_{n\geq 0}$ be the monic matrix orthogonal polynomials with respect to $W$ and $\{Q_n\}_{n\geq 0}$ the sequence of orthonormal polynomials defined by $Q_n= S_nR_n$, with $S_n=\|R_n\|^{-1}$  the inverse of the matrix valued  norm of $R_n$.

 We  have  that these polynomials are eigenfunctions of $D$, with matrix-valued eigenvalues,
 \begin{equation}\label{eigenvalues}
   R_n D=  \Lambda_n R_n, \qquad  Q_n D= \tilde \Lambda_n Q_n,\quad \text{ for all } n\geq 0,
 \end{equation}
 with $\tilde \Lambda_n=S_n\Lambda_nS_n^{-1}$.

 They also satisfy the three term recursion relations 
 \begin{equation}\label{ttrrmonic}
 \begin{split}
    xR_n(x)&=  R_{n+1}+ B_n R_n+ A_n R_{n-1},\\
  xQ_n(x)& = \tilde A_{n+1}^* Q_{n+1}+ \tilde B_n Q_n+ \tilde A_n Q_{n-1},
\end{split}
 \end{equation}
where 
\begin{align*}
A_n & = \|R_n\|^2 \|R_{n-1}\|^{-2},& (B_nS_n)&=(B_nS_n)^*,\\
\tilde A_n &=S_nA_nS_{n-1}^{-1}= \|R_n\| \|R_{n-1}\|^{-1}, & \tilde B_n  &=S_nB_nS_n^{-1},
\end{align*}
here we adopt the convention that $P_{-1}=Q_{-1}=0$.

\

The fact that the symmetry of $D$ implies that we have a bispectral situation  as above  has been established in \cite{GPT03,DG04}, 
where the pairs $(W,D)$ are called ``classical pairs".

 \

 Recall the setup in the section on Preliminaries.

\noindent We fix a natural number $N$ and ${\Omega} \in (a,b)$ and  consider the following operators $\chi_{\Omega} $ and $\chi_N$ in $L^2(W)$:
  $\chi_{\Omega}$ acts on $L^2(W)$ by multiplication by the characteristic function of the interval $(a, {\Omega})$
  and $\chi_N=\mathcal F\tilde \chi_N \mathcal F^{-1}$ is the ``projection" on the (left) module (over the ring of matrices) spanned  by $\{ Q_0, Q_1, \dots, Q_N\}$.
Explicitly,
\begin{equation}\label{chiN}
 \chi_N(f)= \sum_{n=0}^N \langle f, Q_n\rangle Q_n.
\end{equation}

Hence, the band-time-band limiting operator $ EE^*$, that now can be rewritten as $EE^*=\chi_{\Omega} \chi_N\chi_{\Omega}$, is an integral operator acting from the right hand side as in \eqref{intoper},
with kernel
\begin{equation*}
  k(x,y)=\sum_{n=0}^N Q_n^*(x)Q_n(y).
\end{equation*}

 The operator $E^*E$ is the finite dimensional block-matrix given in \eqref{MatrixM}.  Also, now we have that the action of the time-band-time limiting operator $\mathcal F E^*E \mathcal F^{-1} =\chi_N \chi_{\Omega}\chi_N$ is given by
$$\chi_N \chi_{\Omega}\chi_N(f)= \sum_{i=0}^N \left (\int_{a}^{\Omega} f(x)W(x)Q_i^*(x)dx\right) Q_i,$$
for $f\in L^2(W)$.
\

The main result of this section is a simple proof of the existence of a commuting symmetric operator for both of these time and band limiting operators  $EE^*$ and  $\mathcal F E^*E \mathcal F^{-1}$. For this purpose, we will construct an operator $T$ which commutes with {\bf each of} $\chi_N$ and $\chi_{\Omega}$. This important idea already appears in \cite{Perline}. It is also used in the later paper \cite{Walter}.

While this will clearly imply the commutativity with
both  $EE^*$ and  $\mathcal F E^*E \mathcal F^{-1}$ we do not look into the possibility of finding a local operator that commutes with these ones but fails to commute with both $\chi_N$ and $\chi_{\Omega}$.

\bigskip

We  {\bf assume the following hypothesis} on the weight $W$ and the differential operator $D$:  There exists a matrix $M$, independent of the variables $x$, $n$ and the parameter ${\Omega}$, but possibly dependent on $N$, such that
\begin{equation}\label{HYP}
  \Big(M-x (\Lambda_{N+1}+\Lambda_N) \Big)W(x)- W(x)\Big(M-x (\Lambda_{N+1}+\Lambda_N) \Big)^*=0.
\end{equation}

In the expression above the dependence on the differential operator $D$ is hidden in the eigenvalues $\Lambda_N$ of the monic orthogonal polynomials. Explicitly if the differential operator $D$ is of the form $D= \partial ^2 F_2 + \partial  F_1 +F_0$ and we write $F_2=F_{22} x^2+ F_{21} x+F_{20}$, $F_1= F_{11} x+F_{10}$, we have that
\begin{equation}\label{moneig}
\Lambda_n= \Lambda_n(D)=n(n-1)F_{22}+ nF_{11}+F_0.
\end{equation}

 \

From the symmetric differential operator $D$, the eigenvalues of the monic orthogonal polynomials and this matrix  $M$, we build the following differential operator, acting on the ``right-hand side''
\begin{subequations}
\begin{equation}\label{operatorT}
  T= xD+Dx-2{\Omega} D-(\Lambda_{N+1}+\Lambda_N) x+M.
\end{equation}
Let us observe that if $D= \partial ^2 F_2 + \partial  F_1 +F_0$ then
$xD= D x+ 2 \partial \, F_2+F_1.$
 Therefore
\begin{equation}\label{operatorT2}
  \tfrac 12 T= D (x-{\Omega})+ \partial \, F_2(x)+ \tfrac 12 \Big (F_1(x) -x (\Lambda_{N+1}+\Lambda_N)+M\Big).
\end{equation}
\end{subequations}

\smallskip
 \begin{prop}\label{sym}
   The differential operator $T$ is a symmetric operator with respect to $W$, in $[a,b]$ and also in $[a,{\Omega}]$.
 \end{prop}
 \begin{proof}
Since $D$ is symmetric with respect to $W$ in $[a,b]$ it is clear that $xD+Dx$ and $2{\Omega} D$ are also   symmetric operators in $[a,b]$.
	 Hence, from \eqref{operatorT}, for any  smooth enough  functions $f,g\in L^2(W)$ we have
\begin{align*}
\langle f T,g\rangle-\langle f,gT\rangle=\int_{a}^{b} f(x)\left(M-x (\Lambda_{N+1}+\Lambda_N) \right)W(x)- W(x)\left(M-x (\Lambda_{N+1}+\Lambda_N) \right)^*g(x)\, dx.
\end{align*}

Thus, we have that $T$ is a symmetric operator in $[a,b]$ if and only if the operator of order zero $M-x (\Lambda_{N+1}+\Lambda_N)$ satisfies \eqref{HYP}.

Now we will prove that $T$ is symmetric  with respect to $W$ in $[a,{\Omega}]$.

  From \cite{GPT03} or \cite{DG04} we have that a differential operator $D=\frac{d^2}{dx^2} F_2(x)+\frac{d}{dx} F_1(x)+F_0$
is symmetric with respect to a weight $W$ defined in $(a,b)$ if and only if it satisfies, for $a<x<b$, the symmetry equations
\begin{equation} \label{symmeq}
  \begin{split}
  F_2 W & =WF_2^*,\\
   2(F_2W)'-F_1W &=WF_1^*,\\
 (F_2W)''-(F_1W)'+F_0W &=WF_0^*,
  \end{split}
\end{equation}
and  the boundary conditions
\begin{equation}\label{boundary}
  \lim_{x\to a,b} F_2(x)W(x)=0, \quad \lim_{x\to a,b} \big (F_1(x)W(x)-WF_1^*(x)\big)=0.
\end{equation}

We have the following relations among the coefficients of the differential operators $D=\partial ^2 F_2 + \partial  F_1 +F_0$ and $T=\partial ^2 \tilde F_2 + \partial  \tilde F_1 +\tilde F_0 $,
\begin{align*}
  \tilde F_2&= (x-{\Omega}) F_2,\\
  \tilde F_1&= (x-{\Omega}) F_1+ F_2,\\ 
  \tilde F_0&= (x-{\Omega}) F_0+ \tfrac 12 \big(F_1(x) -x (\Lambda_{N+1}+\Lambda_N)+M\big).
\end{align*}

Since $T$ is a symmetric operator with respect to the weight $W$ in the interval $(a,b)$ we have that $\{\tilde F_0,\tilde F_1,\tilde F_2\}$ satisfy \eqref{symmeq} and \eqref{boundary}. Then, to prove that $T$ is symmetric in $(a,{\Omega})$ it suffices to prove that
\begin{equation}
\lim_{x\to {\Omega}} \tilde F_2(x)W(x)=0, \quad \lim_{x\to {\Omega}} \big (\tilde F_1(x)W(x)-W\tilde F_1^*(x)\big)=0.
\end{equation}

Since $D$ is symmetric with respect to the weight $W$ in the interval $(a,b)$ we have that $\{F_0,F_1, F_2\}$ also satisfy \eqref{symmeq}, thus
$$\lim_{x\to {\Omega}} \tilde F_2W= \lim_{x\to {\Omega}} (x-{\Omega}) F_2W=0,$$
and
$$\lim_{x\to {\Omega}} \big (\tilde F_1(x)W(x)-W\tilde F_1^*(x)\big)= \lim_{x\to {\Omega}} \Big ((x-{\Omega})( F_1W-W F_1^*) +F_2W-WF_2^*\Big)=0,$$
completing the proof.
\end{proof}

 \

 \begin{prop}\label{omega}
   The differential operator $T$ commutes with the band-limiting operator $\chi_{\Omega}$.
 \end{prop}
 \begin{proof}
	 Let us observe that $T\chi_{\Omega}=\chi_{\Omega} T$ if and only if $(fT)\chi_{\Omega}=(f\chi_{\Omega})T$, for all  smooth enough  $f\in L^2(W)$.
Since the operator $T$ is symmetric with respect to $W$ in $[a,b]$ and also in $[a, {\Omega}]$ we have
\begin{align*}
  \langle (\chi_{\Omega} f)T, g\rangle & = \langle \chi_{\Omega} f , gT\rangle = \int_a^b \chi_{\Omega}(x) f(x) W(x)(gT)^*(x)\, dx= \int_a^{\Omega}  f(x) W(x)(gT)^*(x)\, dx\\
  & = \langle f, gT\rangle_{\Omega} = \langle fT, g\rangle_{\Omega}= \int_a^{\Omega} (fT)(x)W(x) g^*(x)dx= \int_a^b (fT)(x)\chi_{\Omega}(x)W(x) g^*(x)dx\\
  & =  \langle  fT \chi_{\Omega}, g\rangle,
\end{align*}
	 for  all smooth enough  $f$ and $g$. Hence, $T$ commutes with $\chi_{\Omega}$.
 \end{proof}

\begin{remark}
We observe that if $T$ is a symmetric operator with respect to $W$ in $[a,b]$ then $T$ commutes with $\chi_{\Omega}$ if and only if $T$ is symmetric with respect to $W$ in $[a,{\Omega}]$.
\end{remark}

\begin{prop}\label{Ttridiagonal} 
   For any $n\ge0$, there exist matrices $X_n$, $Y_n$ and $Z_n$ such that
   $$Q_n T= X_n Q_{n+1}+Y_n Q_{n}+ Z_n Q_{n-1}.$$
   Moreover $X_n^*=Z_{n+1}$ and $Y_n^*=Y_n$, with the convention $Q_{-1}=0$.
\end{prop}
\begin{proof}
	For any $n$,   $Q_nT$ is a polynomial of degree $n+1$ or less, since $M$ is a matrix independent of $x$. Hence $  Q_n T= \sum_{j=0}^{n+1}  K_{n,j} Q_j$, for some matrices $\{K_{n,j}\}$.

It is easy to see that, since $T$ is symmetric, we  have
$$\langle Q_nT,Q_j\rangle=\langle Q_n,Q_jT\rangle=0,  \text{ for all } j<n-1.$$
hence $$  Q_n T= \sum_{j=n-1}^{n+1}  K_{n,j} Q_j=X_n Q_{n+1}+Y_n Q_{n}+ Z_n Q_{n-1}.$$

Now we observe that $X_n=\langle Q_nT, Q_{n+1}\rangle=\langle Q_n, Q_{n+1}T\rangle=  Z_{n+1}^*$ and that
$Y_n=\langle Q_nT, Q_{n}\rangle=\langle Q_n, Q_{n}T\rangle=  Y_{n}^*$.
This concludes the proof.
\end{proof}

\begin{cor} \label{blockT} We have 
 $$X_n= \|R_n\|^{-1}\big ( \Lambda_{n+1}+ \Lambda_n-\Lambda_{N+1}- \Lambda_N \big )\|R_{n+1}\|,$$
	where $\{R_n\}_n$ is the sequence of  monic orthogonal polynomials.

 In particular $X_N=Z_{N+1}=0$.
\end{cor}
\begin{proof}
From \eqref{operatorT}, by using the three term recursion relation \eqref{ttrrmonic} and \eqref{eigenvalues}, we have
\begin{equation*}\begin{split}
\langle R_nT,R_{n+1}\rangle &=
\langle (R_n)(xD+Dx-2{\Omega} D-(\Lambda_{N+1}+\Lambda_N) x+M),R_{n+1}\rangle\\
 &= \langle (R_n)(xD+Dx-(\Lambda_{N+1}+\Lambda_N) x),R_{n+1}\rangle\\
 &= \langle R_{n+1}D+R_{n}\Lambda_{n}x-R_{n+1}(\Lambda_{N+1}+\Lambda_N),R_{n+1}\rangle
	\intertext{(and since $\{R_{n}\}$ is the monic sequence of orthogonal polynomials) }
& =\langle R_{n+1}\Lambda_{n+1}+R_{n+1}\Lambda_{n}-(\Lambda_{N+1}+\Lambda_N)R_{n+1},R_{n+1}\rangle\\
&=(\Lambda_{n+1}+\Lambda_n-\Lambda_{N+1}-\Lambda_N)\langle R_{n+1},R_{n+1}\rangle.
\end{split}
\end{equation*}
Hence, by using that $Q_n=\|R_n\|^{-1} R_n$ we get
\begin{align*}
\langle Q_nT,Q_{n+1}\rangle & =\|R_{n}\|^{-1}(\Lambda_{n+1}+\Lambda_n-\Lambda_{N+1}-\Lambda_N)\langle R_{n+1},R_{n+1}\rangle \|R_{n+1}\|^{-1}\\
 & =\|R_{n}\|^{-1}(\Lambda_{n+1}+\Lambda_n-\Lambda_{N+1}-\Lambda_N)\|R_{n+1}\|.
\end{align*}
By Proposition \ref{Ttridiagonal} we know that $\langle Q_nT,Q_{n+1}\rangle=\langle X_nQ_{n+1},Q_{n+1}\rangle=X_n$. Thus, the proof is complete.
\end{proof}

 \begin{prop}\label{ene}
   The differential operator $T$ commutes with  the time-limiting operator $\chi_N$.
 \end{prop}
 \begin{proof}
	 Let  $f$ be a smooth enough function  in $L^2(W)$, by using Proposition \ref{Ttridiagonal}, the fact that $T$ is symmetric and the explicit expression in \eqref{chiN} we have
\begin{align*}
  fT\chi_N& = \sum_{n=0}^N \langle fT, Q_n\rangle Q_n 
 =  \sum_{n=0}^N \Big( \langle f, Q_{n+1}\rangle  X_n^*+\langle f, Q_{n}\rangle  Y_n^*+ \langle f, Q_{n-1}\rangle  Z_n^* \Big)Q_n.
\end{align*}

On the other hand,
\begin{align*}
  (f\chi_N) T & = \sum_{n=0}^N \langle f, Q_n\rangle Q_nT = \sum_{n=0}^N \langle f, Q_n\rangle \left(X_n Q_{n+1}+Y_n Q_{n}+ Z_n Q_{n-1}\right)\\
  & = \sum_{n=1}^{N+1} \langle f, Q_{n-1}\rangle X_{n-1} Q_n+ \sum_{n=0}^{N} \langle f, Q_{n}\rangle Y_{n} Q_n
  +\sum_{n=0}^{N-1} \langle f, Q_{n+1}\rangle Z_{n+1} Q_n,
\intertext{by Corollary \ref{blockT}  $X_N=0$, thus }
   (f\chi_N) T & = \sum_{n=0}^N \Big( \langle f, Q_{n-1}\rangle  X_{n-1}+\langle f, Q_{n}\rangle  Y_n+ \langle f, Q_{n+1}\rangle  Z_{n+1} \Big)Q_n.
 \end{align*}
 Now the proposition follows from  the fact that $X_n^*=Z_{n+1}$ and $Y_n^*=Y_n$, see Corollary \ref{blockT}.
 \end{proof}
\
\begin{thm}\label{main}
The second order differential operator $T$ is symmetric and commutes with the  time-band-limiting operators $EE^*$ and $\mathcal F E^*E \mathcal F^{-1}$.
\end{thm}
\begin{proof}
The symmetry of $T$ is proved in \ref{sym}.
Recalling that  $EE^*=\chi_{\Omega} \chi_N\chi_{\Omega}$ and $\mathcal F E^*E \mathcal F^{-1} =\chi_N \chi_{\Omega}\chi_N$, the proof follows from Proposition \ref{omega} and Proposition \ref{ene}.
\end{proof}

So far the operators $D$, $S$, $T$ act in $L^2(W)$. Conjugating with  $\mathcal F$ you get difference operators acting in $\ell^2(M_R,\NN_0)$. If we define $$L=\mathcal F^{-1} T \mathcal F,$$
the following result is straightforward.

\begin{cor}
The difference operator $L$ is given by a tridiagonal hermitian semi-infinite matrix, with $R\times R$-block entries, and it commutes with the time-band-limiting operators $ \mathcal F^{-1}EE^*\mathcal F$ and $ E^*E$. The operator $L$, in the standard basis of $\ell^2(M_R,\NN_0)$, is explicitly given by
$$L=\left( \begin{matrix} Y_0& X_0^*&0&0&0&\cdots\\ X_0&Y_1& X_1^*&0&0&\cdots\\ 0&X_1&Y_2& X_2^*&0&\cdots
\\ 0&0&X_2&Y_3& X_3^*&\cdots
\\ 0&0&0&X_2&Y_4&\cdots
\\ \vdots&\vdots&\vdots&\vdots&\vdots&\ddots
\end{matrix}\right),$$
with $X_j$ and $Y_j$ given in Proposition \ref{Ttridiagonal} and Corollary \ref{blockT}.
\end{cor}

\begin{remark}
 From Corollary \ref{blockT} it is clear that $L$ breaks into two blocks, an upper-left block of size $(N+1)\times(N+1)$ yielding a matrix such as the one displayed in \cite{GPZ1} and a lower-right block which is semi-infinite.
\end{remark}


\section{Examples}\label{ex}
\subsection{Scalar cases}

In the scalar case condition \eqref{HYP} is automatically satisfied. For several examples of a commuting differential operator given by \eqref{operatorT} one can see \cite{GVZ1_Heun}.

\subsection{Matrix Gegenbauer weight}
In \cite{PZ16} we study $2\times 2$ matrix-valued orthogonal polynomials associated with spherical functions in the $q$-dimensional sphere $S^q$ (  originally $q$ was a natural number, but these results were later extended to any real positive number). The weight matrix, depending on parameters $0<p<q$, is given by
$$W(x)= (1-x^2)^{\frac q2 -1} \begin{pmatrix}
 px^2+q-p & -qx\\ -qx& (q-p)x^2+p
\end{pmatrix}, \qquad x\in [-1,1].
$$
In this case there exist four linearly independent symmetric differential operators of degree two in the algebra $D(W)$, namely $D_1, D_2, E_3$ and $E_4$. See Section 5 in \cite{PZ16}, and the last paragraph in this example.

In \cite{GPZ1}, \cite{GPZ2} and \cite{CGPZ} we considered the time-band limiting operators $E^*E$  and $EE^*$ for this example. For given $N$ and $\Omega$, we found
a symmetric tridiagonal matrix $L$, with simple spectrum, commuting with the block matrix $E^*E$  and a self-adjoint differential operator $\tilde D$ commuting with the integral operator $EE^*$.

\smallskip
The results on the present  paper give an unified way to obtain such a commuting operators in both situations: Starting with a symmetric differential operator of order two, we search for a matrix $M$ such that condition \eqref{HYP} is satisfied and we build up the operator $T$ by the formula \eqref{operatorT}.

%

\

The monic orthogonal polynomials $\{R_n\}$ are eigenfunctions of the differential operators $D_1$ and $D_2$, whose eigenvalues
 are respectively 
$$  \Lambda_n(D_1)=
  \left(\begin{smallmatrix}
  (n+p)(n+q-p+1) & 0\\ 0& 0
\end{smallmatrix}\right) \quad \text{and } \quad
\Lambda_n(D_2) = \left(\begin{smallmatrix}
 0 & 0\\ 0& (n+p+1)(n+q-p)
\end{smallmatrix}\right).
$$

For the differential operators $D_1$ and $D_2$,  the matrices $M_1$ and $M_2$ given by
$$M_1=\tfrac{(-2N(N+p+1)(N+q-p+1)+q-2p)}{q-2p}\begin{pmatrix}
  0&q-p\\p&0
\end{pmatrix},$$
$$M_2=\tfrac{(2N(N+p+1)(N+q-p+1)+q-2p)}{q-2p}\begin{pmatrix}
  0&q-p\\p&0
\end{pmatrix},$$
satisfy the requirement that
$$ \Big(M_1-x (\Lambda_{N+1}(D_1)+\Lambda_N(D_1)) \Big)W(x) \quad \text{and} \quad \Big(M_2-x (\Lambda_{N+1}(D_2)+\Lambda_N(D_2)) \Big)W(x)$$
are symmetric matrices and therefore they give two differential operators, $T_1$ and $T_2$, commuting with the time and band limiting operators.

The differential commuting operator $\widetilde D$, given in \cite{GPZ2}, is a scalar combination of $T_1+T_2$ and the identity, namely $T_1+T_2=-2\widetilde D+2\Omega(q-p)$. Notice that for $T_1+T_2$ the expression \eqref{operatorT} involves a matrix $M$ that does not depend on $N$, namely
$$M=M_1+M_2=-2\begin{pmatrix}
  0&q-p\\ p&0 \end{pmatrix}.$$

On the other hand, the matrices $L_1,L_2,L_3$ given in \cite{GPZ1} are in the span of $\{\mathcal F^{-1} T_1\mathcal F,\mathcal F^{-1} T_2\mathcal F,I\}$. Furthermore,
$L_1$ and $L_2$ scalar multiples of $\mathcal F^{-1} T_2\mathcal F $ and $\mathcal F^{-1} T_1 \mathcal F$, respectively, and
$$L_3=\frac{p(q+p+1)}{q+2}(L_1+L_2).$$
In \cite{GPZ1} it is proved that $L_1$ and $L_2$ have simple spectrum.

\

  It is worth to notice that for the symmetric differential operators  $E_3$ and $E_4$ in \cite{GPZ2} there is no matrix $M$ satisfying condition \eqref{HYP}.
 This phenomenon, namely that given a weight $W$
one should look at the
algebra $D(W)$ introduced in \cite{CG06}, and for each differential operator in it see if a matrix $M$ satisfying
condition \eqref{HYP} exist will reappear later in example $(4.4)$. When $M$
exists our general result yields a commuting operator $T$.

\subsection{Completing the proof of the result stated in \cite{CG7}}
In \cite{CG7} one looks at matrix valued polynomials which are orthogonal
in the interval $[0,1]$ with respect to the weight density matrix originating in \cite{PT,PR} and given by
\[
W(x) = (1-x)^{\alpha} x^{\beta}
\begin{pmatrix}
\beta+1-kx & (\beta+1-k)x \\
(\beta+1-k)x & (\beta+1-k)x^2
\end{pmatrix}.
\]



The monic orthogonal polynomials $R_n(x)$, associated to this weight $W$ are eigenfunctions of the symmetric differential operator
\[
D = \partial ^2 x (1-x) + \partial(C - xU) - V
\]
with
\[
C = \begin{pmatrix}
\beta+1 & 1 \\
0 & \beta+3
\end{pmatrix},\ U = \begin{pmatrix}
\alpha+\beta+3 & 0 \\
0 & \alpha+\beta+4
\end{pmatrix},
V = \begin{pmatrix}
0 & 0 \\
k-\beta-1 & \alpha+\beta+2-k
\end{pmatrix}.
\]
\smallskip
This operator acts on the right  and we have
$R_nD = \Lambda_nR _n $,
with
\[
\Lambda_n = \begin{pmatrix}
-n(\alpha+\beta+n+2) & 0 \\
 1+\beta-k & -(n+1)(\alpha+\beta+n+2)+k
\end{pmatrix}.
\]

%
%
%
%
%
%
%
The main take home message in \cite{CG7} is that the differential operator
${\widetilde D}$ given by
\[
{\widetilde D} = (x-\Omega)D - \frac {d}{dx} x(1-x) + \mathcal{P}_N(x),
\]
%
%
with
\[
\mathcal{P}_N(x) =
\begin{pmatrix}
x(N^2 + (\alpha+\beta+3)(N+1)) & \alpha+\beta+N+2 \\
x(k-\beta-1) & x(N^2+(\alpha+\beta+4)N+2\alpha+2\beta-k+6)+\beta
\end{pmatrix}\,,
\]
commutes with the integral operator  $S$ given by \eqref{intoper}.




%
%
%
%

The argument given in \cite{CG7} consists in verifying certain identities depending on an index $n$. These have been checked with the use of the computer algebra package Maxima up to very large values of $n$, but no analytical proof is given. We  will see below  that the results above complete the arguments in \cite{CG7}.

%

%
%
%

One can see that the matrix 
\[
M = \begin{pmatrix}
1+\beta & 2(\alpha+\beta)+2N+5 \\
 0 & 3(1+\beta)
\end{pmatrix},
\]
is such that
$$(M-x(\Lambda_{n+1}+\Lambda_n)) W(x) $$
is a symmetric matrix. Therefore the assumption in \eqref{HYP} is verified
and one can check that the commuting operator in \cite{CG7}  is given according to the recipe in \eqref{operatorT2}
$$  \widetilde D= D (x-{\Omega})+ \partial \, F_2(x)+ \tfrac 12 \Big (F_1(x) -x (\Lambda_{N+1}+\Lambda_N)+M\Big).$$

\subsection{An example violating condition \eqref{HYP}}

Consider the matrix valued polynomials which are orthogonal
in the interval $[0,1]$ with respect to the weight density matrix originating in 
 \cite{DdI}, Section 3.3, with parameters $\alpha=\beta=0$, $\kappa=1/2$, $t_0=0$ and given by
\[
W(x) =
\begin{pmatrix}
1+x^2 & 1-x \\
1-x & (1-x)^2
\end{pmatrix}.
\]



We have that
$$D_+ = \partial^2  \begin{pmatrix}
2(x^2-x) & 2x \\
0 & 0
\end{pmatrix} + \partial \begin{pmatrix}
8x-7 & 7-x \\x-1 & 1
\end{pmatrix}
+\frac 12 \begin{pmatrix} 3 & -5  \\
1 & 3
\end{pmatrix}$$
is a  symmetric differential operator with respect to $W(x)$,  ($\phi^+=1$ in the notation of \cite{DdI}),

%
%
%
The monic orthogonal polynomials $R_n(x)$ satisfy $R_n(x)D_+=\Lambda_n(D_+) R_n(x)$, where the eigenvalues are given by
\[
\Lambda_n(D_+)= \begin{pmatrix}
2n^2+6n+3/2 & -n-5/2 \\
n+1/2 & -3/2
\end{pmatrix},
\]
see \eqref{moneig}.
One can check that in this case condition \eqref{HYP} is not satisfied.

\

 We have ample evidence that, for a given $N$ and $\Omega$, the corresponding
 integral operator commutes with the differential one given by
$$\widetilde D = \partial^2 x(x-1) (x-\Omega) + \partial X + Y$$
with
\begin{align*}
X &= \begin{pmatrix}
5x^2-4 \Omega x-4x +3 \Omega & 2(x-\Omega)  \\
0 & 5x^2-4 \Omega x -2x + \Omega
\end{pmatrix},
\\
Y &= \begin{pmatrix}
 \Omega/ 2-3-N(N+4)x & (\Omega+5)/2 \\
(\Omega-1)/2 & -N(N+4)x- \Omega/ 2
\end{pmatrix}.
\end{align*}

Clearly this differential operator does not have the form advertised in \eqref{operatorT}.  We will see below that our explicit construction yields an interesting result.

\

The weight matrix $W(x)$ admits another symmetric differential operator (with $\phi^-=1/3$)
 $$D_- = \partial^2
 \begin{pmatrix}
0 & 2x \\0 & 2x(1-x)\end{pmatrix}
+ \partial \begin{pmatrix}
-1 & 3-x \\x-1 & -8x+3
\end{pmatrix}
+\frac 12 \begin{pmatrix} 5 & -3  \\
3 & -5
\end{pmatrix}.$$
For $D_-$ condition \eqref{HYP} is, once again, not satisfied.

Nevertheless for the symmetric differential operator
$$\tfrac 12 \left(D_+-D_-\right)= \partial ^2 x(x-1)+ \partial \begin{pmatrix}
4x-3 & 2 \\0 & 4x+1
\end{pmatrix}
+\frac 12 \begin{pmatrix} -1 & -1  \\
-1 & 1
\end{pmatrix}, $$
condition \eqref{HYP} is satisfied with $M=\begin{pmatrix}
  3&-3\\1& -1
\end{pmatrix}. $

We observe that in this case, the eigenvalues of the monic polynomials $R_n$ are given by $$\Lambda_n =n(n+3)+ \frac 12 \begin{pmatrix}
  1&1\\1& -1 \end{pmatrix}. $$
Now it easy to verify that the differential operator $\tilde D$ above is exactly the differential operator $T$ given in \eqref{operatorT} for $D=\tfrac 12 \left(D_+-D_-\right)$, therefore it   commutes with the integral operator $EE^*$ (see Theorem \ref{main}).

\smallskip

 At the end of example $(4.2)$ we alluded to the phenomenon seen above: our method can be applied to some of the operators in
the algebra $D(W)$ but not necessarily to all of them. When the algebra has several generators this increases our chances of being able to use our construction. The next example features a case when there is only one generator of order two.

\subsection{An example showing that bispectrality may not be enough to produce a commuting differential operator}

In this section we discuss an example with a  behavior  quite different
from the ones seen so far.  This example has appeared in \cite{CG06}.

The weight density on the real line is given by
\[W(x)=e^{-x^2-2 x}
\begin{pmatrix}
e^{4 x}+x^2 & x \\
 x & 1
\end{pmatrix}.
\]
This weight gives rise to a bispectral family of polynomials and as observed in \cite{CG06} the algebra of differential operators going with this weight has just one generator of order two.  See also \cite{DG07}.
\bigskip 	

One can easily check that
condition \eqref{HYP}
does not hold in this case for   the operator of order two that generates the algebra.

\bigskip

One could still be able to produce,
for each value of the parameters $N,\Omega$,
a (non-trivial) symmetric second order differential operator that would commute with the kernel
\begin{equation*}
  k_N(x,y)=\sum_{n=0}^N Q_n^*(x)Q_n(y),
\end{equation*}
acting on $(-\infty,\Omega]$, even if this operator is not given by the nice prescription for $T$ above.

\bigskip

We have plenty of evidence that such an operator {\bf does not exist}, at least if we insist that our operator should have polynomial coefficients (this is the case of all known examples so far).  Some of this evidence is described below.

\bigskip

We postulate a commuting symmetric second order differential operator of the form
$$D= \partial ^2 F_2 + \partial  F_1 +F_0,$$
where we allow $F_0,F_1,F_2$ to be polynomials of degree not higher than SIX.

By imposing the necessary condition

$$ k_N(x,y)^* D_x= ( k_N(x,y) D_y )^*,$$
see \eqref{diffkernel}, we deduce that with arbitrary constants $r_1,r_2,r_3$ one has
 \[
F_2(x) = \begin{pmatrix}
(N r_2 -r_1)/(2N) & r_1 x/(2 N) \\
 0 & r_2/2
\end{pmatrix},
\]
 as well as
\[
F_1(x) = \begin{pmatrix}
(N r_2-r_1)(1-x)/N & -(r_1 x^2+2 N r_2 x - r_1 x -N r_2)/N \\
 0 & -r_2 (x+1)
\end{pmatrix},
\]
and finally
\[
F_0(x) = \begin{pmatrix}
-r_1+r_3 & r_1 x - r_2 \\
 0 & r_2 + r_3
\end{pmatrix}.
\]
When we look at one of the boundary conditions, we get that up to a nonzero scalar the value of
$$F_2(\Omega) W(\Omega) $$
is given by
\[
\begin{pmatrix}
( (N r_2-r_1) e^{4 \Omega} + N r_2 \Omega^2)/(2 N)   & r_2 \Omega /2 \\
 r_2 \Omega /2 & r_2/2
\end{pmatrix},
\]
and from here it follows that $r_1,r_2$ both vanish. This implies that $D$
is a scalar multiple of the identity.

\smallskip
 We have not given a proof that a nontrivial commuting differential operator with more complicated coefficients may not exist. However we are confident that this is the case, since looking
at the finite dimensional block-matrix given by
$E^*E$ we can verify that the only block-tridiagonal matrix that commutes with it is the identity matrix.

\end{document}